\documentclass[11pt,reqno]{amsart}
\usepackage{amscd,amssymb,amsmath,amsthm}
\usepackage[arrow,matrix]{xy}
\usepackage{graphicx}
\usepackage{cite}
\newtheorem{thm}[subsection]{Theorem}
\newtheorem{lemma}[subsection]{Lemma}

\numberwithin{equation}{section} 

\newcommand{\Q}{\mathbb{Q}}

\newcommand{\C}{\mathbb{C}}

\newcommand{\bc}{\begin{center}}
\newcommand{\ec}{\end{center}}
\newcommand{\be}{\begin{equation}}
\newcommand{\ee}{\end{equation}}
\newcommand{\bea}{\begin{eqnarray}}
\newcommand{\eea}{\end{eqnarray}}
\newcommand{\ba}{\begin{array}}
\newcommand{\ea}{\end{array}}

\def\l{\lambda}
\def\g{\gamma}
\def\r{\rho}

\def\m{\mu}
\def\n{\nu}

\def\Q{\mathbb{Q}}

\def\N{\mathbb{N}}
\def\C{\mathbb{C}}

\begin{document}

\title[On a non-linear $p$-adic dynamical system]{On a non-linear $p$-adic dynamical system}

\author{Rozikov U.A., Sattarov I.A.}

 \address{U.\ A.\ Rozikov \\ Institute of mathematics,
29, Do'rmon Yo'li str., 100125, Tashkent, Uzbekistan.} \email
{rozikovu@yandex.ru}
 \address{I.A. Sattarov\\ Namangan state university, Namangan, Uzbekistan.} \email
{iskandar1207@rambler.ru}

\begin{abstract}

We investigate the behavior of trajectories of a
$(3,2)$-rational $p$-adic dynamical system in the complex $p$-adic
filed $\C_p$, when there exists a unique fixed point $x_0$.
We study this $p$-adic dynamical system by dynamics of real
radiuses of balls (with the center at the fixed point $x_0$). We show that there exists a radius $r$ depending on parameters of the rational function such that: when $x_0$ is an attracting point then the trajectory
of an inner point from the ball $U_r(x_0)$ goes to $x_0$ and each sphere with a radius $>r$ (with the center at $x_0$) is invariant;
When $x_0$ is a repeller point then the trajectory
of an inner point from a ball $U_r(x_0)$ goes forward to the sphere $S_r(x_0)$. Once the trajectory reaches the sphere, in the next step it either goes back to the interior of $U_r(x_0)$ or stays in $S_r(x_0)$ for some time and then goes back to the interior of the ball. As soon as the
trajectory goes outside of $U_r(x_0)$ it will stay (for all the rest of time) in the sphere (outside of $U_r(x_0)$) that it reached first.

\end{abstract}

 \maketitle

\section{Introduction}
What is the main difference between real and $p$-adic space-time? It is the Archimedean axiom. According to this axiom any given large segment on a stright line can be surpassed by successive addition of small segments along the same line. This axiom is valid in the set of real numbers and is not valid in the field of $p$-adic numbers $\Q_p$. However, it is a physical axiom which concerns the process of measurement. To exchange a number field $R$ to $\Q_p$ is the same as to exchange axiomatics in quantum physics (\cite{K}, \cite{VV}).

The representation of $p$-adic numbers by sequences of digits gives a possibility
to use this number system for coding of information. Therefore $p$-adic
models can be used for the description of many information processes. In
particular, they can be used in cognitive sciences, psychology and sociology.
Such models based on $p$-adic dynamical systems \cite{A1}, \cite{A2}.

The study of $p$-adic dynamical systems arises in Diophantine
geometry in the constructions of canonical heights, used for
counting rational points on algebraic varieties over a number
field, as in \cite{CS}. There most recent monograph on $p$-adic dynamics is Anashin
and Khrennikov, \cite{AK09}; nearly a half of Silverman's monograph \cite{Sil07}
also concerns $p$-adic dynamics.

Here are areas where $p$-adic dynamics proved to be effective: computer
science (straight line programs), numerical analysis and simulations
(pseudorandom numbers), uniform distribution of sequences, cryptography
(stream ciphers, $T$-functions), combinatorics (Latin squares),
automata theory and formal languages, genetics. The monograph
\cite{AK09} contains the corresponding survey. For a newer results see recent
papers and references therein: \cite{ARS},  \cite{AKY11}, \cite{Ana10}, \cite{Ana12}, \cite{FL11}, \cite{KMa},  \cite{KLPS09} - \cite{Pin09}, \cite{SAL12}. Moreover,
there are studies in computer science and
cryptography which along with mathematical physics stimulated in
1990-th intensive research in $p$-adic dynamics since it was observed
that major computer instructions (and therefore programs composed
of these instructions) can be considered as continuous transformations
with respect to the 2-adic metric, see \cite{Ana94}, \cite{Ana98}.

In this paper we investigate the behavior of trajectories of an
arbitrary $(3,2)$-rational $p$-adic dynamical system in complex
$p$-adic filed $\C_p$. The paper is organized as follows: in
Section 2 we give some preliminaries. Section 3 contains the
definition of the $(3,2)$-rational function and main results about behavior of trajectories of the $p$-adic dynamical system.

\section{Preliminaries}

\subsection{$p$-adic numbers}

Let $\Q$ be the field of rational numbers. The greatest common
divisor of the positive integers $n$ and $m$ is denotes by
$(n,m)$. Every rational number $x\neq 0$ can be represented in the
form $x=p^r\frac{n}{m}$, where $r,n\in\mathbb{Z}$, $m$ is a
positive integer, $(p,n)=1$, $(p,m)=1$ and $p$ is a fixed prime
number.

The $p$-adic norm of $x$ is given by
$$
|x|_p=\left\{
\begin{array}{ll}
p^{-r}, & \ \textrm{ for $x\neq 0$},\\[2mm]
0, &\ \textrm{ for $x=0$}.\\
\end{array}
\right.
$$
It has the following properties:

1) $|x|_p\geq 0$ and $|x|_p=0$ if and only if $x=0$,

2) $|xy|_p=|x|_p|y|_p$,

3) the strong triangle inequality
$$
|x+y|_p\leq\max\{|x|_p,|y|_p\},
$$

3.1) if $|x|_p\neq |y|_p$ then $|x+y|_p=\max\{|x|_p,|y|_p\}$,

3.2) if $|x|_p=|y|_p$ then $|x+y|_p\leq |x|_p$,

this is a non-Archimedean one.

The completion of $\Q$ with  respect to $p$-adic norm defines the
$p$-adic field which is denoted by $\Q_p$ (see \cite{R}, \cite{S}).

The well-known Ostrovsky's theorem asserts that norms
$|x|=|x|_{\infty}$ and $|x|_p$, $p=2,3,5...$ exhaust all
nonequivalent norms on $\Q$. Any $p$-adic number
$x\neq 0$ can be uniquely represented in the canonical series:
$$
x=p^{\gamma(x)}(x_0+x_1p+x_2p^2+...) , \eqno(2.1)
$$
where $\gamma=\gamma(x)\in\mathbb Z$ and $x_j$ are integers,
$0\leq x_j\leq p-1$, $x_0>0$, $j=0,1,2,...$. Observe that in this case $|x|_p=p^{-\g(x)}$.

The algebraic completion of $\Q_p$ is denoted by $\C_p$ and it is
called {\it complex $p$-adic numbers}.  For any $a\in\C_p$ and
$r>0$ denote
$$
U_r(a)=\{x\in\C_p : |x-a|_p\leq r\},\ \ V_r(a)=\{x\in\C_p :
|x-a|_p< r\},
$$
$$
S_r(a)=\{x\in\C_p : |x-a|_p= r\}.
$$

A function $f:U_r(a)\to\C_p$ is said to be {\it analytic} if it
can be represented by
$$
f(x)=\sum_{n=0}^{\infty}f_n(x-a)^n, \ \ \ f_n\in \C_p,
$$ which converges uniformly on the ball $U_r(a)$.

\subsection{Dynamical systems in $\C_p$}

In this section we recall some known facts concerning dynamical
systems $(f,U)$ in $\C_p$, where $f: x\in U\to f(x)\in U$ is an
analytic function and $U=U_r(a)$ or $\C_p$.

Now let $f:U\to U$ be an analytic function. Denote $x_n=f^n(x_0)$,
where $x_0\in U$ and $f^n(x)=\underbrace{f\circ\dots\circ
f}_n(x)$.

Recall some  the standard terminology of the theory of dynamical
systems. If $f(x_0)=x_0$ then $x_0$
is called a {\it fixed point}. The set of all fixed points of $f$
is denoted by Fix$(f)$. A fixed point $x_0$ is called an {\it
attractor} if there exists a neighborhood $V(x_0)$ of $x_0$ such
that for all points $y\in V(x_0)$ it holds
$\lim\limits_{n\to\infty}y_n=x_0$. If $x_0$ is an attractor then
its {\it basin of attraction} is
$$
A(x_0)=\{y\in \C_p :\ y_n\to x_0, \ n\to\infty\}.
$$
A fixed point $x_0$ is called {\it repeller} if there  exists a
neighborhood $V(x_0)$ of $x_0$ such that $|f(x)-x_0|_p>|x-x_0|_p$
for $x\in V(x_0)$, $x\neq x_0$. Let $x_0$ be a fixed point of a
function $f(x)$. The ball $V_r(x_0)$ (contained in $U$) is said to
be a {\it Siegel disk} if each sphere $S_{\r}(x_0)$, $\r<r$ is an
invariant sphere of $f(x)$, i.e. if $x\in S_{\r}(x_0)$ then all
iterated points $x_n\in S_{\r}(x_0)$ for all $n=1,2\dots$.  The
union of all Siegel desks with the center at $x_0$ is said to {\it
a maximum Siegel disk} and is denoted by $SI(x_0)$.

 In complex geometry, the center of a disk
is uniquely determined by the disk, and different fixed points
cannot have the same Siegel disks. In non-Archimedean geometry, a
center of a disk is a point which belongs to the disk. Therefore,
different fixed points may have the same Siegel desk.

Let $x_0$ be a fixed point of an analytic function  $f(x)$. Put
$$
\l=\frac{d}{dx}f(x_0).
$$

The point $x_0$ is {\it attractive} if $0\leq |\l|_p<1$, {\it
indifferent} if $|\l|_p=1$, and {\it repelling} if $|\l|_p>1$.

\section{$(3,2)$-rational $p$-adic dynamical systems with a unique fixed point}

A function is called a $(n,m)$-rational function if and only if it
can be written in the form $f(x)={P_n(x)\over Q_m(x)}$, where
$P_n(x)$ and $Q_m(x)$ are polynomial functions with degree $n$ and
$m$ respectively, $Q_m(x)$ is not the zero polynomial.

In this paper we consider the dynamical system associated with the
$(3,2)$-rational function $f:\C_p\to\C_p$ defined by
\begin{equation}\label{3.1}
f(x)=\frac{x^3+ax^2+bx+c}{x^2+ax+d}, \ \ a,b,c,d\in \C_p, b\ne d\ \
\end{equation}

where  $x\neq \hat x_{1,2}=\frac{-a\pm\sqrt{a^2-4d}}{2}$.

Denote
$$\mathcal P=\{x\in \C_p: \exists n\in \N, f^n(x)=\hat x_{1,2}\}.$$

The function $f$ has a unique fixed point $x_0={c\over d-b}$.

For any $x\in \C_p\setminus\mathcal P$, by simple calculations we get
\begin{equation}\label{ff}
    |f(x)-x_0|_p=|x-x_0|_p\cdot {|((x-x_0)+(x_0+{a\over 2}))^2+\alpha(x_0)|_p\over |((x-x_0)+(x_0+{a\over 2}))^2+\beta(x_0)|_p},
\end{equation}
where
$$\alpha(x)={8x^4+16ax^3+6(a^2+d)x^2+(6ad+a^3-c)x+a^2d+bd-ac\over 4(x^2+ax+d)}.$$

$$\beta(x)={2x^2+2ax+d+{a^2\over 4}}.$$
For $$\alpha=|\alpha(x_0)|_p, \ \ \beta=|\beta(x_0)|_p,$$
and
$$ \delta=|x_0+{a\over 2}|_p$$
we have that $\alpha$, $\beta$ and $\delta$ satisfy one of the following relations:

I.$\delta\leq \min\{\sqrt{\alpha},\sqrt{\beta}\}$.\ \

II.$\sqrt{\alpha}\leq \min\{\delta,\sqrt{\beta}\}$.\

III.$\sqrt{\beta}\leq \min\{\sqrt{\alpha},\delta\}$.

Consider the following functions:

For $\delta\leq \min\{\sqrt{\alpha},\sqrt{\beta}\}$ define the functions $\varphi_{i}: [0,+\infty)\to [0,+\infty),(i=1,2,3,4,5.)$ as the following

1. If $\delta<\sqrt{\alpha}<\sqrt{\beta}$, then
$$\varphi_{1}(r)=\left\{\begin{array}{lllll}
{\alpha\over\beta}r, \ \ {\rm if} \ \ r<\sqrt{\alpha},\\[2mm]
\alpha^*, \ \ {\rm if} \ \ r=\sqrt{\alpha},\\[2mm]
{r^3\over\beta}, \ \ {\rm if} \ \ \sqrt{\alpha}<r<\sqrt{\beta},\\[2mm]
\beta^*, \ \ {\rm if} \ \ r=\sqrt{\beta},\\[2mm]
r, \ \ \ \ {\rm if} \ \ r>\sqrt{\beta},
\end{array}
\right.
$$
where $\alpha^*$ and $\beta^*$ are some given numbers with $\alpha^*\leq{\alpha\sqrt{\alpha}\over\beta}$, $\beta^*\geq\sqrt{\beta}$.

2. If $\delta<\sqrt{\alpha}=\sqrt{\beta}$, then
$$\varphi_{2}(r)=\left\{\begin{array}{lllll}
r, \ \ \ \ {\rm if} \ \ r\neq\sqrt{\alpha},\\[2mm]
\hat\alpha, \ \ \ \ {\rm if} \ \ r=\sqrt{\alpha},
\end{array}
\right.
$$
where $\hat\alpha$ is a given number.

3. If $\delta=\sqrt{\alpha}<\sqrt{\beta}$, then
$$\varphi_{3}(r)=\left\{\begin{array}{lllll}
\lambda r, \ \ {\rm if} \ \ r<\delta,\\[2mm]
\delta^*, \ \ {\rm if} \ \ r=\delta,\\[2mm]
{r^3\over\beta}, \ \ {\rm if} \ \ \delta<r<\sqrt{\beta},\\[2mm]
\beta^*, \ \ {\rm if} \ \ r=\sqrt{\beta},\\[2mm]
r, \ \ \ \ {\rm if} \ \ r>\sqrt{\beta},
\end{array}
\right.
$$
where $\lambda\leq {\delta^2\over\beta} <1$, $\delta^*\leq{\delta^3\over\beta}$ and $\beta^*\geq\sqrt{\beta}$.

4. If $\delta<\sqrt{\beta}<\sqrt{\alpha}$, then
$$\varphi_{4}(r)=\left\{\begin{array}{lllll}
{\alpha\over\beta}r, \ \ {\rm if} \ \ r<\sqrt{\beta},\\[2mm]
\beta^*, \ \ {\rm if} \ \ r=\sqrt{\beta},\\[2mm]
{\alpha\over r}, \ \ {\rm if} \ \ \sqrt{\beta}<r<\sqrt{\alpha},\\[2mm]
\alpha^*, \ \ {\rm if} \ \ r=\sqrt{\alpha},\\[2mm]
r, \ \ \ \ {\rm if} \ \ r>\sqrt{\alpha},
\end{array}
\right.
$$
where $\alpha^*\leq{\sqrt{\alpha}}$, $\beta^*\geq{\alpha\over\sqrt{\beta}};$.

5. If $\delta=\sqrt{\beta}<\sqrt{\alpha}$, then
$$\varphi_{5}(r)=\left\{\begin{array}{lllll}
\lambda r, \ \ {\rm if} \ \ r<\delta,\\[2mm]
\delta^*, \ \ {\rm if} \ \ r=\delta,\\[2mm]
{\alpha\over r}, \ \ {\rm if} \ \ \delta<r<\sqrt{\alpha},\\[2mm]
\alpha^*, \ \ {\rm if} \ \ r=\sqrt{\alpha},\\[2mm]
r, \ \ \ \ {\rm if} \ \ r>\sqrt{\alpha},
\end{array}
\right.
$$
where $\lambda\geq {\delta^2\over\beta} >1$, $\delta^*\geq{\alpha\over\delta}$ and $\alpha^*\leq\sqrt{\alpha}$.

For $\sqrt{\alpha}\leq \min\{\delta,\sqrt{\beta}\}$ define the function $\phi_{j}: [0,+\infty)\to [0,+\infty)$ $(j=1,2,3)$ as the following

1. If $\sqrt{\alpha}<\delta<\sqrt{\beta}$, then
 $$\phi_{1}(r)=\left\{\begin{array}{lllll}
{\delta^2\over\beta}r, \ \ {\rm if} \ \ r<\delta,\\[2mm]
\delta', \ \ {\rm if} \ \ r=\delta,\\[2mm]
{r^3\over\beta}, \ \ {\rm if} \ \ \delta<r<\sqrt{\beta},\\[2mm]
\beta', \ \ {\rm if} \ \ r=\sqrt{\beta},\\[2mm]
r, \ \ \ \ {\rm if} \ \ r>\sqrt{\beta},
\end{array}
\right.
$$
where $\delta'$ and $\beta'$ some positive numbers with $\delta'\leq{\delta^3\over\beta}$, $\beta'\geq\sqrt{\beta}$.

2. If $\sqrt{\alpha}<\delta=\sqrt{\beta}$, then
 $$\phi_{2}(r)=\left\{\begin{array}{lllll}
\lambda r, \ \ {\rm if} \ \ r<\delta,\\[2mm]
\delta', \ \ {\rm if} \ \ r=\delta,\\[2mm]
r, \ \ \ \ {\rm if} \ \ r>\delta,
\end{array}
\right.
$$
where $\lambda\geq{1}$, $\delta'\leq\delta$.

3. If $\sqrt{\alpha}\leq\delta<\sqrt{\beta}$, then
$$\phi_{3}(r)=\left\{\begin{array}{lllll}
r, \ \ \ \ {\rm if} \ \ r\neq\delta,\\[2mm]
\hat\delta, \ \ \ \ {\rm if} \ \ r=\delta,
\end{array}
\right.
$$
where $\hat\delta$ is a given number.

For $\sqrt{\beta}\leq \min\{\sqrt{\alpha},\delta\}$ we define the function $\psi_{k}: [0,+\infty)\to [0,+\infty)$ $(k=1,2,3)$ as the following

1. If $\sqrt{\beta}<\delta<\sqrt{\alpha}$, then
$$\psi_{1}(r)=\left\{\begin{array}{ll}
{\alpha\over\delta^2}r, \ \ {\rm if} \ \ r<\delta,\\[2mm]
\delta^*, \ \ {\rm if} \ \ r=\delta,\\[2mm]
{\alpha\over r}, \ \ {\rm if} \ \ \delta<r<\sqrt{\alpha},\\[2mm]
\alpha^*, \ \ {\rm if} \ \ r=\sqrt{\alpha},\\[2mm]
r, \ \ \ \ {\rm if} \ \ r>\sqrt{\alpha},
\end{array}
\right.
$$
where $\delta^*\geq{\alpha\over\delta}$ , $\alpha^*\leq\sqrt\alpha$.

2. If $\sqrt{\beta}<\delta=\sqrt{\alpha}$ and $|\delta^2+\alpha|_p<\delta^2$, then $$\psi_{2}(r)=\left\{\begin{array}{ll}
\lambda r, \ \ {\rm if} \ \ r<\delta,\\[2mm]
\hat\delta, \ \ {\rm if} \ \ r=\delta,\\[2mm]
r, \ \ \ \ {\rm if} \ \ r>\delta,
\end{array}
\right.
$$
where $\hat\delta\geq\delta$ , $\lambda={|\delta^2+\alpha|_p\over\delta^2}$.

3. If $\sqrt{\beta}<\sqrt\alpha\leq\delta$ and $|\delta^2+\alpha|_p=\delta^2$, then $$\psi_{3}(r)=\left\{\begin{array}{ll}
r, \ \ \ \ {\rm if} \ \ r\neq\delta,\\[2mm]
\hat\delta, \ \ \ {\rm if} \ \ r=\delta,
\end{array}
\right.
$$
where $\hat\delta$ is a given number.

Using the formula (\ref{ff}) we easily get the following:

\begin{lemma}\label{lf} If $x\in S_r(x_0)$, then the following formula holds
$$|f^n(x)-x_0|_p=\left\{\begin{array}{lllllllllll}
\varphi_{1}^n(r), \ \ \mbox{if} \ \ \delta<\sqrt{\alpha}<\sqrt{\beta}\,\\[2mm]
\varphi_{2}^n(r), \ \ \mbox{if} \ \ \delta<\sqrt{\alpha}=\sqrt{\beta}\,\\[2mm]
\varphi_{3}^n(r), \ \ \mbox{if} \ \ \delta=\sqrt{\alpha}<\sqrt{\beta}\,\\[2mm]
\varphi_{4}^n(r), \ \ \mbox{if} \ \ \delta<\sqrt{\beta}<\sqrt{\alpha}\,\\[2mm]
\varphi_{5}^n(r), \ \ \mbox{if} \ \ \delta=\sqrt{\beta}<\sqrt{\alpha}\,\\[2mm]
\phi_{1}^n(r), \ \ \mbox{if} \ \ \sqrt{\alpha}<\delta<\sqrt{\beta}\,\\[2mm]
\phi_{2}^n(r), \ \ \mbox{if} \ \ \sqrt{\alpha}<\delta=\sqrt{\beta}\,\\[2mm]
\phi_{3}^n(r), \ \ \mbox{if} \ \ \sqrt{\alpha}\leq\sqrt{\beta}<\delta\,\\[2mm]
\psi_{1}^n(r), \ \ \mbox{if} \ \ \sqrt{\beta}<\delta<\sqrt{\alpha},\\[2mm]
\psi_{2}^n(r), \ \ \mbox{if} \ \ \sqrt{\beta}<\delta=\sqrt{\alpha},\\[2mm]
\psi_{3}^n(r), \ \ \mbox{if} \ \ \sqrt{\beta}<\sqrt{\alpha}\leq\delta.
\end{array}\right.$$
\end{lemma}
Thus the $p$-adic dynamical system $f^n(x), n\geq 1, x\in \C_p\setminus \mathcal P$ is
related to the real dynamical systems generated by $\varphi_{i}$, $\phi_{j}$ and $\psi_k$. Now we are going to study these (real) dynamical systems.

\begin{lemma}\label{l1a} The dynamical system generated by $\varphi_{i}(r),(i=1,2,3,4,5)$ has the following properties:
\begin{itemize}
\item[1.] ${\rm Fix}(\varphi_{i})=\{0\}\cup$
$$\left\{\begin{array}{lll}
\{r: r>\sqrt{\beta}\} \cup\{\beta^*:\ \ \mbox{if} \ \ \sqrt{\beta}=\beta^*\}, \ \ \mbox{for} \ \ \alpha<\beta, \ $i$=1,3,\\[2mm]
\{r: r\neq\sqrt{\alpha}\} \cup\{\hat\alpha:\ \ \mbox{if} \ \ \hat\alpha=\sqrt{\alpha}\}, \ \ \mbox{for} \ \ \alpha=\beta,\ \ $i$=2,\\[2mm]
\{r: r>\sqrt{\alpha}\} \cup\{\alpha^*:\ \ \mbox{if} \ \ \sqrt{\alpha}=\alpha^*\}, \ \ \mbox{for} \ \ \alpha>\beta, \ $i$=4,5.\\
\end{array}
\right.;$$

\item[2.] If $\alpha<\beta$, then for functions $\varphi_{i}(r), i=1,3$ we have
$$\lim_{n\to\infty}\varphi_{i}^n(r)=\left\{\begin{array}{lll}
0, \ \ \mbox{for all} \ \ r<\sqrt{\beta},\\[2mm]
r, \ \ \mbox{for all} \ \ r>\sqrt{\beta},\\[2mm]
\varphi_{i}(\beta^*), \ \ \mbox{if} \ \ r=\sqrt{\beta}
\end{array}\right.;
$$

\item[3.] If $\alpha=\beta$, then for function $\varphi_2(r)$
$$\lim_{n\to\infty}\varphi_{2}^n(r)=\left\{\begin{array}{lll}
r, \ \ \mbox{for all} \ \ r\neq\sqrt{\alpha},\\[2mm]
\varphi_2(\hat\alpha), \ \ \mbox{if} \ \ r=\sqrt{\alpha},
\end{array}\right.;
$$
\item[4.] If $\alpha>\beta$, then for functions $\varphi_{i}(r), i=4,5$ we have
$$\lim_{n\to\infty}\varphi_{i}^n(r)=\left\{\begin{array}{lll}
0, \ \ \mbox{if} \ \ \ r=0,\\[2mm]
\in C_i, \ \ \mbox{for all}  \ \ 0<r\leq\sqrt{\alpha},\\[2mm]
r, \ \ \mbox{for all} \ \ r>\sqrt{\alpha},
\end{array}\right.;
$$
where $C_4=\left(\sqrt{\alpha}, \alpha/\sqrt{\beta}\right)\cup \{\beta^*\}$, \ \ $C_5=\left(\sqrt{\alpha}, \alpha/\delta\right)\cup \{\delta^*\}$.
\end{itemize}
\end{lemma}
\begin{proof} 1. This is the result of a simple analysis of the equation $\varphi_{i}(r)=r$.

Proofs of parts 2,3 follow from the property that $\varphi_{i}(r)$, $i=1,2,3$ is an increasing
function (except at points of discontinuity). The part 4 easily can be proved using graphs of the corresponding functions $\varphi_{i}$, $i=4,5$.
\end{proof}
We note that for any $a\in C_i$ there exists $x\in (0,\sqrt{\alpha})$ such that $\varphi_i^m(x)=a$, $i=4,5$.

\begin{lemma}\label{lp} The dynamical system generated by $\phi_{j}(r),(j=1,2,3)$ has the following properties:
\begin{itemize}
\item[A.] ${\rm Fix}(\phi_{j})=\{0\}\cup$
$$\left\{\begin{array}{lll}
\{r: r>\sqrt{\beta}\} \cup\{\beta^*:\ \ \mbox{if} \ \ \beta^*=\beta\}, \ \ \mbox{for} \ \ \sqrt{\alpha}<\delta<\sqrt{\beta},\\[2mm]
\{r: r\neq\delta\} \cup\{\hat\delta:\ \ \mbox{if} \ \ \hat\delta=\delta\}, \ \ \mbox{for} \ \ \sqrt{\alpha}\leq\sqrt{\beta}<\delta,\\[2mm]
\{r: r\neq\delta\} \cup\{\delta^*:\ \ \mbox{if} \ \ \delta^*=\delta\}, \ \ \ \ \mbox{for} \ \ \sqrt{\alpha}<\sqrt{\beta}=\delta \ \ and  \ \ \delta^2=|\delta^2+\beta|_p,\\[2mm]
\{r: r>\delta\} \cup\{\delta^*:\ \ \mbox{if} \ \ \delta^*=\delta\}, \ \ \mbox{for} \ \ \sqrt{\alpha}<\sqrt{\beta}=\delta \ \ and  \ \ \delta^2>|\delta^2+\beta|_p.
\end{array}
\right.;$$

\item[B.] For function $\phi_{1}(r),$ we have
$$\lim_{n\to\infty}\phi_{1}^n(r)=\left\{\begin{array}{lll}
0, \ \ \mbox{for all} \ \ r<\sqrt{\beta},\\[2mm]
r, \ \ \mbox{for all} \ \ r>\sqrt{\beta},\\[2mm]
\phi_{1}(\beta^*), \ \ \mbox{if} \ \ r=\sqrt{\beta}
\end{array}\right.;
$$

\item[C.] For function $\phi_2(r)$:
\begin{itemize}
\item[C.a] If $\delta^2=|\delta^2+\beta|_p$, then
$$\lim_{n\to\infty}\phi_{2}^n(r)=\left\{\begin{array}{lll}
r, \ \ \mbox{for all} \ \ r\neq\delta,\\[2mm]
\phi_2(\delta^*), \ \ \mbox{if} \ \ r=\delta,
\end{array}\right.;
$$
\item[C.b] If $\delta^2>|\delta^2+\beta|_p$, then
$$\lim_{n\to\infty}\phi_{2}^n(r)=\left\{\begin{array}{lll}
0, \ \ \mbox{if} \ \ \ r=0,\\[2mm]
\in B, \ \ \mbox{for all}  \ \ 0<r\leq\delta,\\[2mm]
r, \ \ \mbox{for all} \ \ r>\delta,
\end{array}\right.;
$$
where $B=\left(\delta, \delta^3/|\delta^2+\beta|_p\right)$,
\end{itemize}
\item[D.] For function $\phi_{3}(r),$ we have
$$\lim_{n\to\infty}\phi_{3}^n(r)=\left\{\begin{array}{lll}
r, \ \ \mbox{for all} \ \ \ r\neq\delta,\\[2mm]
\phi_3(\hat\delta), \ \ \mbox{if} \ \ r=\delta
\end{array}\right..
$$
\end{itemize}
\end{lemma}
\begin{proof} The proof consists simple analysis of the functions $\phi_{j}(r)$ using their graphs.
\end{proof}
The following lemma is obvious:

\begin{lemma}\label{lp1} The dynamical system generated by $\psi_{k}(r), (k=1,2,3)$ has the following properties:
\begin{itemize}
\item[(i)] ${\rm Fix}(\psi_{k})=\{0\}\cup$
$$\left\{\begin{array}{ll}
\{r: r>\sqrt{\alpha}\} \cup\{\alpha^*:\, {\rm if}\, \alpha^*=\alpha\}, \, {\rm for} \, \sqrt{\beta}<\delta<\sqrt{\alpha},\\[2mm]
\{r: r>\delta\} \cup\{\hat\delta:\, {\rm if}\, \hat\delta=\delta\}, \, {\rm for} \, \sqrt{\beta}<\delta=\sqrt{\alpha} \ \ and  \ \ \delta^2>|\delta^2+\alpha|_p,\\[2mm]
\{r: r\neq\delta\} \cup\{\hat\delta:\, {\rm if}\, \hat\delta=\delta\}, \ \ \, {\rm for} \, \sqrt{\beta}<\sqrt{\alpha}\leq\delta \ \ and  \ \ \delta^2=|\delta^2+\alpha|_p,
\end{array}
\right.;$$

\item[(ii)] For function $\psi_{1}(r),$ we have
$$\lim_{n\to\infty}\psi_{1}^n(r)=\left\{\begin{array}{lll}
0, \ \ \mbox{if} \ \ \ r=0,\\[2mm]
\in E, \ \ \mbox{for all}  \ \ 0<r\leq\sqrt{\alpha},\\[2mm]
r, \ \ \mbox{for all} \ \ r>\sqrt{\alpha},
\end{array}\right.;
$$
where $E=\left(\sqrt{\alpha}, \alpha/\delta\right)\cup \{\delta^*\}$.
\item[(iii)] For function $\psi_{2}(r)$
$$\lim_{n\to\infty}\psi_{2}^n(r)=\left\{\begin{array}{lll}
0, \ \ \mbox{for all} \ \ \ r<\delta,\\[2mm]
\psi_2(\hat\delta), \ \ \mbox{if}  \ \ r=\delta,\\[2mm]
r, \ \ \mbox{for all} \ \ r>\delta,
\end{array}\right.;
$$
\item[(iv)] For function $\psi_{3}(r)$
$$\lim_{n\to\infty}\psi_{3}^n(r)=\left\{\begin{array}{lll}
r, \ \ \mbox{for all} \ \ \ r\neq\delta,\\[2mm]
\psi_3(\hat\delta), \ \ \mbox{if}  \ \ r=\delta.
\end{array}\right.
$$
\end{itemize}
\end{lemma}

Now we shall apply these lemmas to the study of the $p$-adic dynamical system generated by $f$.

Using Lemma \ref{lf} and Lemma \ref{l1a} we obtain the following

 \begin{thm}\label{t1a} If $\delta\leq\min\{\sqrt{\alpha},\sqrt{\beta}\}$ and $x\in S_r(x_0)$,  then
 the $p$-adic dynamical system generated by $f$ has the following properties:
\begin{itemize}
\item[1.] The following spheres are invariant with respect to $f$:
$$\begin{array}{lll}
S_r(x_0), \ \ \mbox{if} \ \ r>\sqrt{\max\{\alpha,\beta\}} \ \ and \ \ \alpha\neq\beta,\\[2mm]
S_r(x_0), \ \ \mbox{if} \ \ r\neq\sqrt{\alpha}, \,and \ \ \alpha=\beta,
\end{array};$$
\item[2.] For $\alpha<\beta$, we have
$$\lim_{n\to\infty}f^n(x)=x_0, \ \ \mbox{for all} \ \ r<\sqrt{\beta},$$
$$f\left(S_r(x_0)\setminus\mathcal P\right)\subset S_r(x_0), \ \ \mbox{for all} \ \ r>\sqrt{\beta},$$
$$\lim_{n\to \infty}f^n(x)\in \left\{\begin{array}{ll}
S_{\varphi_{1}(\beta^*)}(x_0), \ \ \mbox{if} \ \ \delta<\sqrt{\alpha}, \ \ r=\sqrt{\beta},\\[2mm]
S_{\varphi_{3}(\beta^*)}(x_0), \ \ \mbox{if} \ \ \delta=\sqrt{\alpha}, \ \ r=\sqrt{\beta};
\end{array}\right.
$$
\item[3.] If $\alpha=\beta$, then
$$f\left(S_r(x_0)\setminus\mathcal P\right)\subset S_r(x_0), \ \ \mbox{for all} \ \ r\neq\sqrt{\alpha},$$
$$\lim_{n\to \infty}f^n(x)\in S_{\varphi_{2}(\hat\alpha)}(x_0), \ \ \mbox{if} \ \ r=\sqrt{\alpha};$$
\item[4.] If $\alpha>\beta$, then
$$\lim_{n\to \infty}f^n(x)\in S_{\r}(x_0), \ \r\in \left\{\begin{array}{ll}
C_4, \ \ \mbox{if} \ \ \delta<\sqrt{\beta},\\[2mm]
C_5, \ \ \mbox{if} \ \ \delta=\sqrt{\beta}
\end{array}
\right., \ \
  0<r\leq\sqrt{\alpha},$$
$$f\left(S_r(x_0)\setminus\mathcal P\right)\subset S_r(x_0), \ \ \mbox{for all} \ \ r\neq\sqrt{\alpha}.$$
\end{itemize}
\end{thm}

By Lemma \ref{lf} and Lemma \ref{lp} we obtain the following

\begin{thm}\label{tlp} If $\sqrt{\alpha}\leq\min\{\delta,\sqrt{\beta}\}$ and $x\in S_r(x_0)$, then the $p$-adic dynamical system generated by $f$ has the following properties:
\begin{itemize}

\item[A.] The following spheres are invariant:
$$S_r(x_0), \ \ \mbox{if} \ \  \sqrt{\alpha}<\delta<\sqrt{\beta},\ \ r>\sqrt{\beta},$$
$$S_r(x_0), \ \ \mbox{if} \ \ \sqrt{\alpha}\leq\sqrt{\beta}<\delta,\ \ r\neq\delta,$$
$$S_r(x_0), \ \ \mbox{if} \ \ \sqrt{\alpha}<\sqrt{\beta}=\delta, \ \ \delta^2=|\delta^2+\beta|_p \ \ and \ \ r\neq\delta,$$
$$S_r(x_0), \ \ \mbox{if} \ \ \sqrt{\alpha}<\sqrt{\beta}=\delta, \ \ \delta^2>|\delta^2+\beta|_p \ \ and \ \ r>\delta,$$
\item[B.] For $\sqrt{\alpha}<\delta<\sqrt{\beta}$, we have
$$\lim_{n\to\infty}f^n(x)=x_0, \ \ \mbox{for all} \ \ r<\sqrt{\beta},$$
$$f\left(S_r(x_0)\setminus\mathcal P\right)\subset S_r(x_0), \ \ \mbox{for all} \ \ r>\sqrt{\beta},$$
$$\lim_{n\to \infty}f^n(x)\in S_{\phi_{1}(\beta^*)}(x_0), \ \ \mbox{if} \ \ r=\sqrt{\beta};$$
\item[C.] Let $\sqrt{\alpha}<\sqrt{\beta}=\delta$.
\begin{itemize}

 \item[C.a)] If $\delta^2=|\delta^2+\beta|_p$, then
$$f\left(S_r(x_0)\setminus\mathcal P\right)\subset S_r(x_0),\ \ \mbox{for any}\ \ r\neq \delta,$$
$$\lim_{n\to \infty}f^n(x)\in S_{\phi_{2}(\delta^*)}(x_0), \ \ \mbox{if} \ \ r=\delta;$$

\item[C.b)] If $\delta^2>|\delta^2+\beta|_p$, then
$$\lim_{n\to \infty}f^n(x)\in S_\m (x_0),\ \m\in B, \ \ \mbox{for any} \ \ 0<r\leq\delta,$$
$$f\left(S_r(x_0)\setminus\mathcal P\right)\subset S_r(x_0),\ \ \mbox{for any}\ \ r>\delta;$$
\end{itemize}
\item[D.] If $\sqrt{\alpha}\leq\sqrt{\beta}<\delta$, then
$$f\left(S_r(x_0)\setminus\mathcal P\right)\subset S_r(x_0),\ \ \mbox{for any}\ \ r\neq \delta,$$
$$\lim_{n\to \infty}f^n(x)\in S_{\phi_{3}(\hat\delta)}(x_0), \ \ \mbox{if} \ \ r=\delta.$$
\end{itemize}
\end{thm}

By Lemma \ref{lf} and Lemma \ref{lp1} we get

\begin{thm}\label{tlp1} If $\sqrt{\beta}\leq\min\{\delta,\sqrt{\alpha}\}$, and $x\in S_r(x_0)$, then the dynamical system generated by $f$ has the following properties:
\begin{itemize}
\item[(i)] The following spheres are invariant:
$$S_r(x_0), \ \ if \ \  \delta<\sqrt{\alpha},\ \ r>\sqrt{\alpha},$$
$$S_r(x_0), \, {\rm if} \, \delta=\sqrt{\alpha}, \ \delta^{2}>|\delta^2+\alpha|_p, \ \ r>\delta,$$
$$S_r(x_0), \, {\rm if} \, \delta\geq\sqrt{\alpha}, \ \ \delta^2=|\delta^2+\beta|_p,\ \ r\neq\delta;$$
\item[(ii)] Let $\delta<\sqrt{\alpha}$. Then
$$\lim_{n\to \infty}f^n(x)\in S_\n (x_0),\ \n\in E, \ \ \mbox{for any} \ \ 0<r\leq\sqrt{\alpha},$$
$$f\left(S_r(x_0)\setminus\mathcal P\right)\subset S_r(x_0),\ \ \mbox{for any}\ \ r>\sqrt{\alpha};$$
\item[(iii)] If $\delta=\sqrt{\alpha}$  \ \ and  \ \ $\delta^2>|\delta^2+\alpha|_p$, then
$$\lim_{n\to\infty}f^n(x)=x_0, \ \ \mbox{for all} \ \ r<\delta,$$
$$\lim_{n\to \infty}f^n(x)\in S_{\psi_{2}(\hat\delta)}(x_0), \ \ \mbox{if} \ \ r=\delta;$$
$$f\left(S_r(x_0)\setminus\mathcal P\right)\subset S_r(x_0),\ \ \mbox{for any}\ \ r>\delta;$$
\item[(iv)] If $\delta\geq\sqrt{\alpha}$  \ \ and  \ \ $\delta^2=|\delta^2+\alpha|_p$, then
$$\lim_{n\to \infty}f^n(x)\in S_{\psi_{3}(\hat\delta)}(x_0), \ \ \mbox{if} \ \ r=\delta;$$
$$f\left(S_r(x_0)\setminus\mathcal P\right)\subset S_r(x_0),\ \ \mbox{for any}\ \ r\neq\delta.$$
\end{itemize}
\end{thm}

\section*{ Acknowledgements}

 The first author thanks the Department of Algebra, University of
Santiago de Compostela, Spain,  for providing financial support of
his many visits to the Department. He was also supported by the Grant No.0251/GF3 of Education and Science Ministry of Republic
of Kazakhstan.


\begin{thebibliography}{999}

\bibitem{ARS} S. Albeverio, U.A. Rozikov, I.A. Sattarov. $p$-adic $(2,1)$-rational dynamical systems. {\it Jour. Math. Anal. Appl.} {\bf 398}(2) (2013), 553--566.

\bibitem{A1} S. Albeverio, A. Khrennikov, P.E. Kloeden, Memory retrieval as a $p$-adic dynamical system, {\it BioSys}. {\bf 49} (1999), 105-–115.

\bibitem{A2} S. Albeverio, A. Khrennikov, B. Tirozzi, S. De Smedt, $p$-adic dynamical systems, {\it Theor. Math. Phys}. {\bf 114} (1998), 276--287.

\bibitem{AK09} V. Anashin, A. Khrennikov. {\it Applied Algebraic Dynamics}, V. 49,
de Gruyter Expositions in Mathematics. Walter de Gruyter, Berlin — New York, 2009.

\bibitem{AKY11} V. S. Anashin, A. Yu. Khrennikov, and E. I. Yurova. Characterization of ergodicity of
$p$-adic dynamical systems by using van der Put basis. {\it Doklady Mathematics}, {\bf 83}(3) (2011), 306--
308.

\bibitem{Ana94} V. S. Anashin.  Uniformly distributed sequences of $p$-adic integers. {\it Math. Notes},
{\bf 55}(2) (1994), 109--133.

\bibitem{Ana98} V. S. Anashin.  Uniformly distributed sequences in computer algebra, or how to construct
program generators of random numbers. {\it J. Math. Sci}., {\bf 89}(4) (1998), 1355--1390.

\bibitem{Ana10} V. Anashin.  Non-Archimedean ergodic theory and pseudorandom generators. {\it The Computer
Journal}, {\bf 53}(4) (2010), 370--392.

\bibitem{Ana12} V. Anashin.  Automata finiteness criterion in terms of van der Put series of automata
functions. {\it $p$-Adic Numbers, Ultrametric Analysis and Applications}, {\bf 4}(2) (2012), 151--160.

\bibitem{CS} G. Call and J. Silverman, Canonical height on varieties with morphisms,
{\it Compositio Math}. {\bf 89}(1993), 163-205.

\bibitem{FL11} A.-H. Fan and L.-M. Liao.  On minimal decomposition of $p$-adic polynomial dynamical
systems. {\it Adv. Math.}, {\bf 228} (2011), 2116--2144.

\bibitem{KMa} M. Khamraev, F.M. Mukhamedov, On a class of rational $p$-adic dynamical systems, {\it J. Math. Anal. Appl}. {\bf 315} (1) (2006), 76-–89.

\bibitem{K} A. Yu. Khrennikov, $p$-Adic valued distributions in mathematical physics, Kluwer, Dordrecht (1994).

\bibitem{KLPS09} J. Kingsbery, A. Levin, A. Preygel, and C. E. Silva.  On measure-preserving c1 transformations
of compact-open subsets of non-archimedean local fields. {\it Trans. Amer. Math.
Soc.}, {\bf 361}(1) (2009), 61--85.

\bibitem{KLPS11} J. Kingsbery, A. Levin, A. Preygel, and C. E. Silva.  Dynamics of the $p$-adic shift and
applications. {\it Disc. Contin. Dyn. Syst.}, {\bf 30}(1) (2011), 209–218.

\bibitem{LSY12} D. Lin, T. Shi, and Z. Yang.  Ergodic theory over $\mathbb{F}_2[[X]]$. {\it Finite Fields
Appl.}, {\bf 18} (2012), 473--491.

\bibitem{M1} F.M. Mukhamedov, U.A. Rozikov,  On rational $p$-adic dynamical systems.
{\it Methods of Func. Anal. and Topology.} {\bf 10}(2), 21-31 (2004).

\bibitem{Pin09} J.-E. Pin. {\it Profinite methods in automata theory}. In Symposium on Theoretical Aspects
of Computer Science — STACS 2009, pages 31–50, Freiburg, 2009.

\bibitem{R} A.M. Robert, {\it A course of $p$-adic analysis}, Springer, New York, 2000.

\bibitem{S} W. Schikhof, Ultrametric calculas, Cambridge Univ., Cambridge (1984).

\bibitem{SAL12} T. Shi, V. Anashin and D. Lin. {\it Linear weaknesses in $T$-functions}. In
T. Helleseth and J. Jedwab, editors, SETA 2012, volume 7280 of Lecture Notes Comp.
Sci., pages 279--290, Berlin–Heidelberg, 2012. Springer-Verlag.

\bibitem{Sil07} J. Silverman. {\it The arithmetic of dynamical systems}. Number 241 in Graduate Texts in
Mathematics. Springer-Verlag, New York, 2007.

\bibitem{VV} V.S. Vladimirov, I.V. Volovich, and E.I. Zelenov, The spectral
theory in the p-adic quantum mechanics. {\it Izvestia Akad. Nauk SSSR, Ser.
Mat.}, {\bf 54}(2), 275--302 (1990).


\end{thebibliography}
\end{document}